\documentclass[11pt,a4paper,fleqn]{scrartcl}    
\usepackage[english]{babel}                     
\usepackage[T1]{fontenc}                   
\usepackage[applemac]{inputenc}
\usepackage{graphicx}    
\usepackage{tikz}        
\usepackage{pgfplots}
\pgfplotsset{compat=1.15}
\usepackage{mathrsfs}
\usetikzlibrary{arrows}           %
\usepackage[font=small]{quoting}      
\usepackage{enumitem,wrapfig}
\usepackage[colorlinks=true,linkcolor=blue]{hyperref}
\usepackage{xcolor}
\usepackage{pgf,tikz}
\usepackage{picture}
\usepackage{amsmath,amsthm,amsfonts,amssymb}

\newcommand{\diesis}{^\#}

\newtheorem{theorem}{Theorem}[section]
\newtheorem{lemma}{Lemma}[section]
\newtheorem{prop}{Proposition}[section]
\newtheorem{cor}{Corollary}[section]

\theoremstyle{definition}
\newtheorem{definiz}{Definition}[section]

\newtheorem{rem}{Remark}[section]

\newcommand{\ds}{\displaystyle}

\newcommand{\R}{\mathbb R}

\newcommand{\de}{\partial}
\newcommand{\eps}{\varepsilon}

\begin{document}
\title{On the behaviour of the first eigenvalue of the $p$-Laplacian with Robin boundary conditions as $p$ goes to $1$}
 \author{Francesco Della Pietra$^{*}$ \\ Carlo Nitsch$^{*,**}$\\ Francescantonio Oliva$^{*}$\\ Cristina Trombetti%
       \thanks{Universit\`a degli studi di Napoli Federico II, Dipartimento di Matematica e Applicazioni ``R. Caccioppoli'', Via Cintia, Monte S. Angelo - 80126 Napoli, Italia; 
       $^{**}$Scuola Superiore Meridionale, Universit\`a degli studi di Napoli Federico II, Largo San Marcellino 10, 80138 Napoli, Italy. 
       Email: f.dellapietra@unina.it, carlo.nitsch@unina.it, francescantonio.oliva@unina.it, cristina.trombetti@unina.it }}

%
%
%
%

%

\maketitle

\begin{abstract}
	\textbf{Abstract}. In this paper we study the $\Gamma$-limit, as $p\to 1$, of the functional
	 $$
 J_{p}(u)=\frac{\displaystyle\int_\Omega |\nabla u|^p + \beta\int_{
		\partial \Omega} |u|^p}{\displaystyle \int_\Omega |u|^p},
		$$
		 where  $\Omega$ is a smooth bounded open set in $\R^{N}$, $p>1$ and $\beta$ is a real number. Among our results, for $\beta >-1$, we derive an isoperimetric inequality for 
 		 \[
		 \Lambda(\Omega,\beta)=\inf_{u \in BV(\Omega), u\not \equiv 0}  \frac{\displaystyle |Du|(\Omega) + \min(\beta,1)\int_{
			\partial \Omega} |u|}{\displaystyle \int_\Omega |u|}
		 \]
		   which is the limit as $p\to 1^{+}$ of
$ \lambda(\Omega,p,\beta)= \ds \min_{u\in W^{1,p}(\Omega)} J_{p}(u).	 $ 
We show that among all bounded and smooth open sets with given volume, the ball maximizes $\Lambda(\Omega, \beta)$ when $\beta \in$ $(-1,0)$ and minimizes $\Lambda(\Omega, \beta)$ when $\beta \in[0, \infty)$.
		 
\end{abstract}

\begin{center}
\begin{minipage}{.8\textwidth}
\tableofcontents
\end{minipage}
\end{center}

\section{Introduction}
 In last years a great attention has been devoted to the study of isoperimetric inequalities inolving the following functional
 \begin{equation}
 \label{intro_pb}
 \lambda(\Omega,p,\beta)=
 \min_{u\in W^{1,p}(\Omega)}
 \frac{\displaystyle\int_\Omega |\nabla u|^p + \beta\int_{
		\partial \Omega} |u|^p}{\displaystyle \int_\Omega |u|^p},
 \end{equation}
 where  $\Omega$ is a smooth bounded open set in $\R^{N}$, $p>1$ and $\beta$ is a real number. It is clear that the minimizers $\varphi_{p}$ of \eqref{intro_pb} are solutions to the following Robin boundary value problem:
 \begin{equation*}
	\begin{cases}
		\displaystyle -\Delta_p \varphi_p = \lambda(\Omega,p,\beta) \left|\varphi_p\right|^{p-2}\varphi_p & \text{ in }\Omega,\\
		\displaystyle \left|\nabla \varphi_{p}\right|^{p-2}\frac{\partial \varphi_p}{\de \nu}  + \beta |\varphi_p|^{p-2}\varphi_p=0 & \text{ on } \partial\Omega.
	\end{cases}
\end{equation*}

\medskip

In the study of shape optimization problems for $\lambda(\Omega,p,\beta)$, there is a striking difference between cases $\beta<0$ and $\beta \ge 0$. We first aim to give an idea of the existing literature which is huge for these problems. If $\beta$ is nonnegative, it is widely known that the following Faber-Krahn inequality holds (see \cite{Bo2} for $p=N=2$, \cite{D1} for $p=2$ and $N\ge 2$, \cite{buda,daifu} for any $1<p<+\infty$; see also \cite{pota} for a more general nonlinear setting):
\begin{equation}\label{fkintro}
	\lambda(\Omega,p,\beta)\ge \lambda(\Omega\diesis,p,\beta),
\end{equation}
where $\Omega\diesis$ is the ball, centered at the origin, with the same volume of $\Omega$; namely the ball minimizes $\lambda$ in the class of sets with given volume.

\smallskip

On the contrary, when $\beta<0$ the situation becomes more delicate. Indeed, if $p=2$, it was conjectured in \cite{b} that the ball maximizes $\lambda(\Omega,2,\beta)$ among  smooth bounded domains $\Omega$ of fixed volume. In the same paper this property was initially proved when $N=2$, $|\beta|$ small enough and for \textit{nearly circular domains} which are, roughly speaking, infinitesimal perturbations of a circle.   Later, the conjecture has been disproved in \cite{fk}, in any dimension, for $|\beta|$ great enough and in case of $\Omega$ as a spherical shell. By the way, for bounded planar domains of class $C^2$, the authors show the existence of a value $\beta^*$ such that the ball maximizes the eigenvalue for any $\beta\in [\beta^*,0]$; this is mainly proved by a suitable asymptotic expansion in $\beta$ for the first eigenvalue. Among other things, the former result has been extended in \cite{kp} for any $N\ge 3$ and $p>1$ when $\Omega$ is any simply connected bounded $C^2$ domain once again for $|\beta|$ small enough. Moreover, it is worth mentioning that in \cite{fnt} the authors show the validity of \eqref{fkintro} for any given $\beta<0$ in the class of Lipschitz sets which are, in some sense, close to a ball in the Hausdorff sense.
On the other hand, for any $\beta \le 0$, $\lambda(\Omega,p,\beta)$ is maximized by the ball when $\Omega$ belongs to a suitable class of sets with the same perimeter. In particular, $C^{2}$ planar bounded domains (\cite{afk}) or convex bounded domains in $\R^{N}$
(\cite{bfnt}).
\medskip

In the current paper we deal with the case $p\to1$ in problem \eqref{intro_pb}; as one should expect here the natural associated space is $BV(\Omega)$. As we will see, $\lambda(\Omega,p,\beta)$ converges to
\begin{equation} \label{intro_pb1}
\textcolor{red}	{\Lambda(\Omega,\beta)}= \inf_{u \in BV(\Omega), u\not \equiv 0}  \frac{\displaystyle |Du|(\Omega) + \min(\beta,1)\int_{
			\partial \Omega} |u|}{\displaystyle \int_\Omega |u|}
\end{equation}
where $|Du|(\Omega)$ stands for the total variation of the distributional gradient of $u$ in $\Omega$.
We first prove that problem \eqref{intro_pb1} admits a minimum for any $\beta>-1$ (see Theorem \ref{teo_ex_p=1} below); this result is strengthened by the fact that in case $\beta<-1$ the infimum of the associated functional in \eqref{intro_pb} is $-\infty$.  Latter assertion, roughly speaking, depends on the fact that, while the denominator may tend to zero, the numerator remains strictly negative (see Remark \ref{rem_reg} for the precise computations). However the main result concerns the validity of an isoperimetric inequality for $\Lambda(\Omega,\beta)$ when $\beta>-1$ (see Theorem \ref{teo_fabkra} below). In particular we show that, among all bounded, sufficiently smooth open sets with given volume,  the ball maximizes $\Lambda(\Omega,\beta)$ when $\beta \in ]-1,0]$  and it minimizes $\Lambda(\Omega,\beta)$ when  $\beta \ge 0$.

\medskip
 
The proof of Theorem \ref{teo_fabkra} follows as an application of Proposition \ref{prop_radial} in which the computation of $\Lambda(\Omega,\beta)$ is made explicit  when $\Omega$ is a ball. Here the proof's main ingredient is given by the $\Gamma$-convergence of $J_p$ towards $J$ (defined in \eqref{jp} and \eqref{j} below) and then the convergence of $\lambda(\Omega,p,\beta)$ to $\Lambda(\Omega,\beta)$ as $p\to 1^+$ (see Theorem \ref{gammaconvteo} and Corollary \ref{gammaconvcor} below).      

\medskip

The plan of the paper is the following: in Section \ref{s_firsteigen} we present the minimization problem for $p\ge 1$ while in Section \ref{s_Gamma} we deal with the $\Gamma$-convergence result. In Section \ref{s_FK} we state and prove the isoperimetric inequalities for $\Lambda(\Omega,\beta)$. We conclude with Section \ref{s_Cheeger} by showing a Cheeger type result. 
 
\section{Notations and preliminaries}
Throughout this paper $\Omega$ is a bounded open and connected set of $\mathbb{R}^N$ ($N\ge 2$) with sufficiently smooth boundary (see Remark \ref{rem_reg} below). For a given set $A$ we denote by $\chi_A$ its characteristic function, by $|A|$ its Lebesgue measure, and by $A^{\#}$ the ball centered at the origin such that $|A^{\#}|=|A|$. Moreover, $\mathcal{H}^{N-1}(E)$ will be the $(N-1)$-dimensional Hausdorff measure of a set $E$, and by $B_R$ the ball centered at the origin with radius $R$. 

Let us now recall some basic fact on functions of bounded variation. For a detailed treatment of the subject, we refer the reader, for example, to \cite{afp,eg}. The total variation in $\Omega$ of a function $u\in L^{1}(\Omega)$ is
\[
\left| Du\right|(\Omega) =\sup\left\{\int_{\Omega} u\textrm{div}\phi,\; \phi \in C^{1}_{0}(\Omega,\R^{N}),\; \|\phi\|_{L^{\infty}}\le 1\right\}.
\]
Then $u$ is a function of bounded variation in $\Omega$, and we write $u\in BV(\Omega)$, if $\left| Du\right|(\Omega)$ is finite. In this case, the distributional derivative of $u$ is a finite Radon measure in $\Omega$, that is
\[
\int_{\Omega} u\frac{\de \phi}{\de x_{i}} dx=-\int_{\Omega} \phi d D_{i}u,\quad \forall \phi \in C_{0}^{\infty}(\Omega),\quad i=1,\ldots,N,
\]
for some vector valued measure $Du=(D_{1}u,\ldots,D_{N}u)$.
 The set $BV(\Omega)$ is a Banach space if endowed with the norm
\[
\|u\|_{BV(\Omega)}=\|u\|_{L^{1}(\Omega)}+\left|D u\right|(\Omega).
\]
If the characteristic function $\chi_{E}$ of a set $E\subset\R^{N}$ has bounded variation, we say that $E$ has finite perimeter in $\Omega$, and denote the relative perimeter of $E$ as
\[ 
P_{\Omega}(E)= \left| D\chi_{E}\right|(\Omega).
\]
Finally, we will write $P(E)$ instead of $P_{\R^N}(E)$.

The coarea formula states that if $u\in BV(\Omega)$, then
\[
\left|D u\right|(\Omega)=\int_{-\infty}^{+\infty} P_{\Omega}(u>t)dt.
\]
We finally recall that a sequence $u_{n}$ in $BV(\Omega)$ weak$^{*}$ converges to $u\in BV(\Omega)$ if $u_{n}\to u$ in $L^{1}(\Omega)$ and
\[
\lim_{n\to +\infty}\int_{\Omega}\phi d D_{i} u_{n}=\int_{\Omega} \phi d D_{i}u,\quad \forall \phi \in C_{0}(\Omega),\quad i=1,\ldots,N.
\]

\medskip

Let us recall the definition of the Cheeger constant $h(\Omega)$ for $\Omega\subset\R^N$, which is given by
\begin{equation}\label{cheeger}
	h(\Omega):= \inf_{E\subseteq\Omega} \frac{P(E)}{|E|}=\inf_{u\in BV(\Omega)}\frac{|Du|(\R^{N})}{\int_{\Omega}|u|}.
\end{equation} 
It is well known that there is at least a minimum for problem \eqref{cheeger} when $\Omega$ has Lipschitz boundary. A set for which the minimum is attained is called Cheeger set for $\Omega$. Finally, for subsequent use, we recall that the ball is self-Cheeger, that is if $\Omega = B_R$ for some $R>0$ then the Cheeger set of $\Omega$ is $B_R$ itself and then 
\begin{equation}\label{cheeger2}
	h(B_R)=\frac{P(B_R)}{|B_R|}.
\end{equation}

\medskip

In what follows we will use a trace inequality for $BV$-functions: 
if $\Omega$ is a Lipschitz bounded open set, then 
\begin{equation}\label{trace}
	\int_{\partial\Omega} |v| \le c_1 |Dv| (\Omega)+ c_2\int_\Omega |v|,
\end{equation}
for any $v\in BV(\Omega)$, with $c_1=\sqrt{1+L^2}$ and $L$ is the Lipschitz constant of $\partial\Omega$; actually, if $\Omega$ has $C^1$ boundary then $c_1$ can be chosen as $1+\varepsilon$ for any $\varepsilon>0$, with $c_2$ depending on $\varepsilon$ as well. Furthermore, if $\partial\Omega$ has mean curvature bounded from above then we may choose $c_1=1$. We refer the reader to \cite{anzg,giusti}.
To conclude the section, we recall the following lower semicontinuity result.
\begin{prop}[\cite{modica}]
\label{modicone}
For any $\beta\ge-1$, the functional 
\[
F(u)= \left|D u\right|(\Omega)+ \min\{\beta,1\}\int_{\de \Omega} \left|u\right|
\]
is lower semicontinuous on $BV(\Omega)$ with respect to the topology of $L^{1}(\Omega)$.
\end{prop}

\section{The first eigenvalue problem with Robin boundary conditions}
\label{s_firsteigen}

In this section we highlight some well known features on the Robin first eigenvalue problem in both cases $p>1$ and $p=1$.

\subsection{The case $p>1$}

Let us consider the functional
\begin{equation}\label{jp}
J_p(u):=  \frac{\displaystyle\int_\Omega |\nabla u|^p + \beta\int_{
		\partial \Omega} |u|^p}{\displaystyle \int_\Omega |u|^p}, \  u\in W^{1,p}(\Omega), u\not\equiv 0,
\end{equation}
and let
\begin{equation}
\label{eigjp}
\lambda(\Omega,p,\beta)=\inf_{\varphi\in W^{1,p}(\Omega), \varphi\not \equiv 0} J_p(\varphi).
\end{equation}
The following classical result concerning problem \eqref{eigjp} holds.
\begin{theorem}
Let $\Omega$ be a bounded Lipschitz open set, and let $p>1$. For any $\beta\in\R$, there exists a minimum of problem \eqref{eigjp}. Moreover, if $\Omega$ is connected, then $\lambda(\Omega,p,\beta)$ is simple, that is admits a unique minimizer, up to a multiplicative constant. In this case, the minimizers are positive or negative in $\Omega$ and they are in $C^{1,\alpha}(\Omega)$. Furthermore, they satisfy the following Robin boundary value problem:
\begin{equation*}
	\begin{cases}
		\displaystyle -\Delta_p \varphi_p = \lambda(\Omega,p,\beta) \left|\varphi_p\right|^{p-2}\varphi_p & \text{ in }\Omega,\\
		\displaystyle \left|\nabla \varphi_{p}\right|^{p-2}\frac{\partial \varphi_p}{\de \nu}  + \beta |\varphi_p|^{p-2}\varphi_p=0 & \text{ on } \partial\Omega.
	\end{cases}
\end{equation*}
\end{theorem}
\begin{proof}
As regards the existence of minimizers, it follows by a well-known argument of Calculus of Variations (see for example \cite{kp}). For the simplicity of the eigenvalue and the regularity of the eigenfunctions we can refer the reader, for example, to \cite{buda,mr}.
\end{proof}

\begin{rem}
	We explicitly observe that the case $\beta=0$ is trivial, because for any $p\ge 1$ $\lambda(\Omega,p,\beta)=0$ and the minimizers are the constant functions.
Moreover, when $\beta\to+\infty$, $\lambda(\Omega,p,\beta)$ tends to the first Dirichlet eigenvalue of the $p-$Laplacian.

\end{rem}

\subsection{The case $p=1$}

Here we are mainly interested into the study of the functional
\begin{equation}\label{j}
	J(u):= \frac{\displaystyle |Du|(\Omega) + \min(\beta,1)\int_{
			\partial \Omega} |u|}{\displaystyle \int_\Omega |u|}, \  u\in BV(\Omega), u\not\equiv 0,
\end{equation}
which in case $\beta\ge 1$ is nothing else than (when $u$ is assumed to be identically zero outside $\Omega$)
$$
J(u)= \frac{\displaystyle \left|Du\right|(\mathbb{R}^N)}{\ds\int_\Omega |u|}, \  u\in BV(\Omega), u \not\equiv 0,
$$
and we deal with the related minimization problem
\begin{equation}\label{minJ}
	\Lambda(\Omega,\beta)=\inf_{\varphi\in BV(\Omega), \varphi\not \equiv 0} J(\varphi).
\end{equation}
\begin{rem}
If $\beta>1$, one could ask what happens when studying the problem
\[
\lambda(\Omega,1,\beta)=\inf_{u\in BV(\Omega), u\not \equiv 0}\frac{\displaystyle |Du|(\Omega) + \beta\int_{
			\partial \Omega} |u|}{\displaystyle \int_\Omega |u|}.
\]
Actually, in this case
\[
\lambda(\Omega,1,\beta)=\Lambda(\Omega,\beta)=h(\Omega),
\] 
where $h(\Omega)$ is the Cheeger constant defined in \eqref{cheeger}. Indeed, it is easy to convince that $\lambda(\Omega,1,\beta)\ge h(\Omega)$. Moreover, if $v$ is a minimum for $h(\Omega)$, Theorem $3.1$ of \cite{ls} assures the existence of a sequence $v_k\in C^\infty_c(\Omega)$ which converges to $v$ in  $L^q(\Omega)$ for any $q\le \frac{N}{N-1}$ and such that $\|\nabla v_k\|_{L^1(\Omega)}$ converges to $|Dv|(\mathbb{R}^N)$ as $k\to \infty$. Hence 
\[
\lambda(\Omega,1,\beta) \le \lim_{k\to +\infty} \frac{\int_{\Omega}|\nabla v_{k}|}{\int_{\Omega}|v_{k}|} = h(\Omega).
\]
\end{rem}
	
 Let us observe that the study of the minimization problem given in \eqref{minJ} can be dealt with by limiting $u$ only to characteristic functions. 

\medskip

Thus, for $E \subseteq \Omega$, evaluating $J$ on the characteristic function of $E$ one has    
\[
J(\chi_E) = \frac{P_\Omega(E) + \min(\beta,1)\mathcal{H}^{N-1}(\partial\Omega \cap\partial E)}{|E|},
\]
and from here on we denote by
\[
R(E,\beta):= J(\chi_E). 
\]
Now we consider the minimization problem given by
\begin{equation}\label{elle}
	\displaystyle \ell(\Omega,\beta):=\inf_{E\subseteq\Omega} R(E,\beta).
\end{equation}	

Depending on $\beta$, the infimum in \eqref{minJ} can be proven to be achieved 
under different conditions on the regularity of the set $\Omega$. 

Therefore, before stating the existence of a minimum for \eqref{minJ}, we need to highlight the regularity of $\partial\Omega$ which is required throughout the paper. 

\begin{rem}[Assumptions on $\Omega$ and on $\beta$]\label{rem_reg}
The boundary of $\Omega$ is required to have Lipschitz regularity when $\beta$ is nonnegative. When $\beta$ is negative, the regularity of $\partial\Omega$ is mainly related to a suitable use of the trace inequality. In particular, through the paper in the case$-1<\beta<0$, we assume $\partial \Omega \in C^1$. For the case $\beta\le-1$, let us observe that the minimum problem for $J$ is not well posed; namely one can show that, for any fixed Lipschitz bounded open set $\Omega$, if $\beta<-1$ then
	\[
	\inf_{u\in BV(\Omega), u\not\equiv 0}J(u)=-\infty.
	\]
	Indeed there exists a sequence $\Omega_\varepsilon\subset\Omega$ such that $P_\Omega(\Omega_\varepsilon) \le P(\Omega) + \varepsilon$ for any $\varepsilon >0$ and such that $|\Omega\setminus \Omega_{\varepsilon}| \to 0$ as $\varepsilon\to 0$ (see for example Theorem $1.1$ of \cite{sch}, as well as the references therein contained).\\  
	Hence one obtains
	\[J(\chi_{\Omega\setminus \Omega_{\varepsilon}}) = \frac{P_\Omega(\Omega_\varepsilon) + \beta P(\Omega)}{|\Omega\setminus \Omega_{\varepsilon}|}\le \frac{(1+\beta)P(\Omega) + \varepsilon}{|\Omega\setminus \Omega_{\varepsilon}|},
	\]
	which goes to $-\infty$ as $\varepsilon \to 0$. 
	
	In the case $\beta = -1$, when $\Omega$ has bounded mean curvature then $\Lambda(\Omega,-1) = \ell(\Omega,-1)$ and they are finite.
	
	We observe that the functional
	\[
 J(\chi_E) = \frac{P_\Omega(E) -\mathcal{H}^{N-1}(\partial\Omega \cap\partial E)}{|E|},
\]
describes the free energy of a liquid which occupies a region $E$ inside a container $\Omega.$ 
If the energy functional is minimized under a volume constraint $|E|= m$, then a minimizer exists (see Theorem 19.5  in  \cite{ma}). Moreover  such a minimizer, denoted by $F$, has constant distributional mean curvature in $\Omega$ and by the Young's law it satisfies  (see Theorems 19.7 and 19.8 in \cite {ma})
$< \nu_F, \nu_{\Omega}> =-1$ on $ \partial(\Omega \cap \partial F)$. 
The above considerations ensure that if one considers a rounded square in the plane the minimum in \eqref{elle} is not achieved.
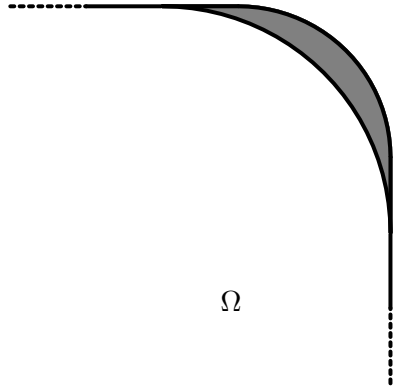
\begin{figure}[h]
\begin{center}
\begin{tikzpicture}[line cap=round,line join=round,>=triangle 45,x=1cm,y=1cm,scale=1]
\draw [line width=1.5pt] (4,0)-- (4,2);
\draw [line width=1.5pt] (2,4)-- (0,4);
\draw [shift={(2,2)},line width=1.5pt]  plot[domain=0:1.5707963267948966,variable=\t]({1*2*cos(\t r)+0*2*sin(\t r)},{0*2*cos(\t r)+1*2*sin(\t r)});
\draw [shift={(1,1)},line width=1.5pt]  plot[domain=0:1.5707963267948966,variable=\t]({1*3*cos(\t r)+0*3*sin(\t r)},{0*3*cos(\t r)+1*3*sin(\t r)});
\draw [shift={(2,2)},line width=1.5pt,color=gray,fill=gray,fill opacity=1]  (0,0) --  plot[domain=0:1.5707963267948966,variable=\t]({1*2*cos(\t r)+0*2*sin(\t r)},{0*2*cos(\t r)+1*2*sin(\t r)}) -- cycle ;
\draw [shift={(2,2)},line width=1.5pt,color=black]  plot[domain=0:1.5707963267948966,variable=\t]({1*2*cos(\t r)+0*2*sin(\t r)},{0*2*cos(\t r)+1*2*sin(\t r)});
\fill[line width=1.5pt,color=gray,fill=gray,fill opacity=1] (2,2) -- (4,2) -- (4,1) -- (1,1) -- (1,4) -- (2,4) -- cycle;
\draw [shift={(1,1)},line width=1pt,color=white,fill=white,fill opacity=1]  (0,0) --  plot[domain=0:1.5707963267948966,variable=\t]({1*3*cos(\t r)+0*3*sin(\t r)},{0*3*cos(\t r)+1*3*sin(\t r)}) -- cycle;
\draw [shift={(1,1)},line width=1.5pt,color=black,fill opacity=1]    plot[domain=0:1.5707963267948966,variable=\t]({1*3*cos(\t r)+0*3*sin(\t r)},{0*3*cos(\t r)+1*3*sin(\t r)});
\draw [line width=1.5pt] (4,2)-- (4,1);
\draw [dotted,line width=1.5pt] (4,-1)-- (4,0);
\draw [line width=1.5pt] (1,4)-- (2,4);
\draw [dotted,line width=1.5pt] (-1,4)-- (0,4);
\draw (1.9017399301080857,.0964282073499274) node {$\Omega$};
\end{tikzpicture}
\end{center}
\caption{When $\beta=-1$ and $\Omega$ is a rounded square of the plane, the sequence of minimizers of $\Lambda(\Omega,-1)$ is given by the sets bounded by a circle of radius $R$ contained in $\Omega$ and tangent to the square, and one of the rounded corners (the grey domain is one of such sets). When the radius of this circle approaches the radius of the circle in the corner, $J$ approaches to $\lambda$.}
\end{figure}

 \end{rem}

\newpage
We are in position to state the existence result for \eqref{minJ}:
\begin{theorem}\label{teo_ex_p=1}
	Let $\beta> -1$. Then there exists a minimum to problem \eqref{minJ}. In particular it holds
	\begin{equation}
		\label{char0}
		\Lambda(\Omega,\beta)=\ell(\Omega,\beta).
	\end{equation}
	Moreover, if $v\in BV(\Omega)$ is a minimizer of \eqref{minJ}, then
	\begin{equation}
		\label{char3}
		\Lambda(\Omega,\beta)=R(\{v>t\},\beta)
	\end{equation}
	for some $t\in\R$.
\end{theorem}
\begin{proof}
 Let $u_n\in BV(\Omega)$ be a minimizing sequence to \eqref{minJ}, with $\|u_n\|_{L^1(\Omega)}=1$.
 
 Let $\beta>0$. In this case $u_n$ is bounded in $BV(\Omega)$, then $u_n$ converges (up to a subsequence) weak$^*$ in $BV$ and strongly  in $L^1(\Omega)$ to $u\in BV(\Omega)$ . Since $u_n$ and $u$ are assumed to be identically zero outside $\Omega$, 
in the case $\beta\ge1$,  $J(u_n)=|D u_n|(\R^N)$ and the lower semicontinuity of the total variation gives
\[
J(u) \le \liminf_n J(u_n)
\]
hence $u$ is a minimum for the functional $J$.

When  $0<\beta<1$, then given $\delta>0$ and $\Omega_\delta=\{x\in\Omega\colon d(x,\de\Omega)>\delta\}$, where $d(\cdot,\de\Omega)$ is the distance function to the boundary of $\Omega$ we have
\[
J(u_n) \ge  |D u_n|(\Omega_\delta) +\beta\left[|D u_n| (\Omega\setminus \Omega_\delta)+\int_{\de\Omega}|u_n|\right]
\]
and by lower semicontinuity
\[
\liminf_{n} J(u_n) \ge |D u|(\Omega_\delta)+\beta |Du|(\R^N\setminus\Omega_\delta).
\]
Hence, being $u\in BV(\Omega)$ for $\delta\to0$ it holds
\[
\liminf_{n} J(u_n) \ge J(u)
\]
and $u$ is a minimum for $J$.

Now let us suppose $-1<\beta<0$. 
We have $J(u_n)\le C$; then using \eqref{trace} with $c_1=1+\eps$, for $\eps$ sufficiently small we have
\[
J(u_n) \ge (1+\beta+\beta\eps)  |D u_n|(\Omega) + c_2\beta  \ge c_2 \beta;
\]
moreover using again \eqref{trace} we get
\[
 |D u_n|(\Omega) \le C-\beta (1+\eps)  |D u_n|(\Omega) -\beta c_2,
\]
and for $\eps$ small this implies that $u_n$ is bounded in $BV(\Omega)$.

Finally, the functional $J$ is lower semicontinuous by Proposition \ref{modicone}. This allows to conclude that $u$ is a minimum of $J$.

Now we show \eqref{char0}. Obviously, we have 
	\[
	\Lambda(\Omega,\beta) \le \ell(\Omega,\beta).
	\]
	In order to show the reverse inequality, let $v\in BV(\Omega)$ be a minimizer of \eqref{minJ}. Then by the coarea formula 
	\[
	 |Dv| (\Omega)= \int_{-\infty}^{+\infty} P_\Omega(\{v>t\})\ dt,
	\]
 one obtains
	\begin{equation} \label{char1}
		\begin{aligned}
			\Lambda(\Omega,\beta)=J(v) &=   \frac{ \displaystyle \int_{-\infty}^{+\infty} P_\Omega(\{v>t\})\ dt + \min(\beta,1)\int_{-\infty}^{+\infty} \mathcal{H}^{n-1}(\{v>t\}\cap \partial \Omega)\ dt}{\displaystyle \int_{-\infty}^{+\infty}|\{v>t\}|\ dt}
			\\
			& = \frac{\displaystyle \int_{-\infty}^{+\infty} R(\{v>t\},\beta)|\{v>t\}|dt}{\displaystyle \int_{-\infty}^{+\infty}|\{v>t\}|\ dt} \ge \inf_{E\subseteq \Omega} R(E,\beta) = \ell(\Omega,\beta), 
		\end{aligned}
	\end{equation} 
	which shows that 
	\begin{equation}
		\label{char2}
		\Lambda(\Omega,\beta) = \ell(\Omega,\beta).
	\end{equation}
	In particular, combining \eqref{char1} and \eqref{char2} we have that
	\[
	\int_{-\infty}^{+\infty}\Big\{R(\{v>t\},\beta)-\ell(\Omega,\beta)\Big\}|\{v>t\}|\,dt=0,
	\]
	and the integrand is nonnegative by definition of $\ell$. Hence, being $v\not\equiv0$, \eqref{char3} holds.
\end{proof}

\section{$\Gamma$-convergence}
\label{s_Gamma}

In this section we are concerned with the convergence and the relevant insight that one gains by taking $p \to 1^+$  in the functional $J_{p}$ defined in \eqref{jp}. Once that $J_p$ is extended to $\infty$ for functions belonging to $BV(\Omega)\setminus W^{1,p}(\Omega)$, we will show that the $\Gamma$-limit as $p\to 1^+$ for $J_p$ is the functional $J$.

\medskip

For the sake of completeness, we firstly precise the notion of $\Gamma$-convergence in our setting. 
\begin{definiz}\label{defconv}
The functional $J_p$ $\Gamma$-converges to $J$ as $p\to 1^+$ in the weak$^*$ topology of $BV(\Omega)$ if, for any $u\in BV(\Omega)$, it holds
\begin{itemize}  	
	\item[i)] \textit{liminf inequality}:  for any sequence $u_p \in BV(\Omega)$ which converges to $u$ weak$^*$ in $BV(\Omega)$ as $p\to 1^+$, then
  	\begin{equation}\label{liminf}
  		\liminf_{p\to 1^+} J_p(u_p) \ge J(u);
  	\end{equation}
  	\item[ii)] \textit{limsup inequality}: there exists a sequence $u_p \in W^{1,p}(\Omega)$ which converges to $u$  weak$^*$ in $BV(\Omega)$ as $p\to 1^+$ such that
  	\begin{equation}\label{limsup}
  		\limsup_{p\to 1^+} J_p(u_p) \le J(u).
  	\end{equation}
  \end{itemize}
\end{definiz}

\begin{theorem}\label{gammaconvteo}
Let $\beta > -1$ then $J_p$ $\Gamma$-converges to the functional $J$ in the sense of Definition \ref{defconv}.  
\end{theorem}
\begin{proof}
For any $u\in BV(\Omega)$ we have to show inequalities \eqref{liminf} and \eqref{limsup} where, without loss of generality, one can assume that $||u_p||_{L^p(\Omega)} = 1$. 
 It is clear that in order to show \eqref{liminf} we can assume that 
the $\liminf$ is finite (otherwise the inequality is trivially satisfied). Hence we can suppose that $u_p$ is in $W^{1,p}(\Omega)$ for any $p>1$. 
\smallskip
\noindent 
We split  the proof in several cases depending   to the value of $\beta$.

\medskip
\textbf{Case $\beta \ge 1$.} We first show \eqref{limsup}. Let us note that Theorem $3.1$ of \cite{ls} assures the existence of a sequence $u_k\in C^\infty_c(\Omega)$ which converges to $u$ in  $L^q(\Omega)$ for any $q\le \frac{N}{N-1}$ and such that $\|\nabla u_k\|_{L^1(\Omega)}$ converges to $|Du|(\mathbb{R}^N)$ as $k\to \infty$. Hence, for any $k>0$, one clearly has that $\|\nabla u_k\|_{L^p(\Omega)}$ converges to $\|Du_k\|_{L^1(\Omega)}$. This also implies the existence of a subsequence $p_k \to 1^+$ as $k\to\infty$ such that 
$\|\nabla u_k\|^{p_k}_{L^{p_k}(\Omega)}$ converges to $|Du|(\mathbb{R}^N)$ as $k\to \infty$. This is sufficient to take $k\to \infty$ in $J_{p_k}(u_k)$ deducing \eqref{limsup}.

\smallskip

Now let us focus on \eqref{liminf}; then $u_p \in W^{1,p}(\Omega)$ is any given sequence which converges to $u$ weak$^*$ in $BV(\Omega)$ as $p\to 1^+$. It follows from the H\"older inequality, a convexity argument and from the fact that $\beta \ge 1$, that one yields to
\[
\begin{aligned}
	\left(\int_{\Omega} |\nabla u_p| + \int_{\partial \Omega} |u_p|\right)^p &\le  \left(\left(\int_{\Omega} |\nabla u_p|^p\right)^{\frac{1}{p}}|\Omega|^{1-\frac{1}{p}} + \left(\int_{\partial \Omega} |u_p|^p\right)^{\frac{1}{p}}|\partial \Omega|^{1-\frac{1}{p}}\right)^p
	\\
	&\le 2^{p-1}\left(\int_{\Omega} |\nabla u_p|^p|\Omega|^{p-1} + \beta\int_{\partial \Omega} |u_p|^p|\partial \Omega|^{p-1}\right)
	\\
	&
	\le 2^{p-1}\max{(|\Omega|^{p-1},|\partial\Omega|^{p-1})} J_p(u_p). 
\end{aligned}
\]
Now since  $2^{p-1}\max{(|\Omega|^{p-1},|\partial\Omega|^{p-1})}$ tends to one as $p\to 1^+$, one deduces from lower semicontinuity that
\[
\liminf_{p\to 1^+} J_p(u_p) \ge |D u|(\Omega) + \int_{\partial \Omega} |u| = |D u|(\mathbb{R}^N) = J(u),
\]
which is \eqref{liminf}.

\medskip

\textbf{Case $0\le\beta<1$.} In order to show \eqref{limsup} we observe that it follows from Theorem $3.9$ of \cite{afp} that there exists a sequence $u_k\in C^\infty(\Omega)$ strictly converging to $u$ in the $BV$-norm, namely $u_k$ converges strongly to $u$ in $L^1(\Omega)$ as $k\to \infty$ and $\|\nabla u_k\|_{L^1(\Omega)}$ converges to $|D u|(\Omega)$ as $k\to\infty$. 
Moreover it follows from Theorem $3.88$ of \cite{afp} that $u_k$ converges to $u$ in $L^1(\partial\Omega, \mathcal{H}^{N-1})$.
Reasoning as for the case $\beta \ge 1$ one deduces that, as $p_k \to 1^+$, $\|\nabla u_k\|^{p_k}_{L^{p_k}(\Omega)}$ converges to $|Du|(\Omega)$ as $k\to \infty$ and  $\|u_k\|^{p_k}_{L^{{p_k}}(\partial\Omega,\mathcal{H}^{N-1})}$ converges to $\|u\|_{L^{1}(\partial\Omega,\mathcal{H}^{N-1})}$ as $k\to \infty$. This means that $u_k$ is the sequence which allows to deduce \eqref{limsup}.

\smallskip

To obtain \eqref{liminf}, let us observe that it follows from the Young inequality that
\begin{equation*}
	\begin{aligned}
	J_p(u_p) &= \int_\Omega |\nabla u_p|^p + \beta\int_{\partial \Omega} |u_p|^p 
	\\
	 &\ge \int_{\Omega} |\nabla u_p| + \beta\int_{\partial \Omega} |u_p| -\frac{p-1}{p} \left(|\Omega| + |\partial\Omega|\right),
	\end{aligned}
\end{equation*} 
where $u_p$ is any sequence converging weak$^*$ to $u$ in $BV(\Omega)$ as $p\to 1^+$.
Then inequality \eqref{liminf} is a direct application of Proposition $1.2$ of \cite{modica}.

\medskip

\textbf{Case $-1< \beta<0$.}  
In this case the inequality \eqref{limsup} can be deduced as for the case $0\le\beta<1$. To show \eqref{liminf} let $u_p \in W^{1,p}(\Omega)$ be a sequence converging to $u$ weak$^*$ in $BV(\Omega)$ as $p\to1^+$ and let us denote by
$$\Omega_\delta:= \{x\in\Omega: d(x,\partial\Omega)>\delta\}, \ \ \Omega'_\delta := \Omega \setminus \Omega_\delta.$$
Let $\psi$ be a smooth function which is zero on $\Omega_\delta$, one on $\partial\Omega$ and such that $|\nabla \psi|\le c\delta^{-1}$ for some positive constant $c$. 
An application of \eqref{trace} to $v=(u-|u_p|^{p-1}u_p)\psi$ gives that 
\begin{equation}\label{trace2}
	\begin{aligned}
	\int_{\partial\Omega} |u-|u_p|^{p-1}u_p| &\le c_1 |D(u-|u_p|^{p-1}u_p)| (\Omega'_\delta)+ (c_2+ c_1c\delta^{-1})\int_{\Omega'_\delta}|u-|u_p|^{p-1}u_p|
	\\
	&\le c_1 |Du|(\Omega'_\delta) + c_1\int_{\Omega'_\delta} |\nabla |u_p|^{p-1}u_p| 
	\\
	&+ (c_2 + c_1c\delta^{-1})\int_{\Omega'_\delta}|u-|u_p|^{p-1}u_p|
	,
	\end{aligned}
\end{equation}
where $\delta$ has been chosen such that $|D(u-|u_p|^{p-1}u_p)|  (\partial\Omega_\delta)= 0$ for any $p>1$ which is admissible since $u-|u_p|^{p-1}u_p \in BV(\Omega)$.\\
Now, using that $|a-b|\ge ||a|-|b||$, it follows from \eqref{trace2} that  
\begin{equation}\label{betaneg1}
	\begin{aligned}
	J(u)-J_p(u_p) &=  |Du| (\Omega) - \int_\Omega |\nabla u_p|^p + \beta\int_{\partial \Omega} \left(|u| - |u_p|^p\right) 
	\\
	&\le  |Du| (\Omega)- \int_\Omega |\nabla u_p|^p + |\beta|c_1 |Du|(\Omega'_\delta)
	\\
	& + |\beta|c_1\int_{\Omega'_\delta} |\nabla |u_p|^{p-1}u_p| 
	\\
	&+ |\beta|(c_2+ c_1c\delta^{-1})\int_{\Omega'_\delta}|u-|u_p|^{p-1}u_p|=:A.
	\end{aligned}
\end{equation}
It follows from Remark \ref{rem_reg} and from the subsequent  discussion  the trace inequality \eqref{trace} on the value of the constants $c_1$ and $c_2$, one can fix $c_1= |\beta|^{-1}$. Hence a simple calculation takes to 
\begin{equation}\label{betaneg2}
\begin{aligned}
A &= A -  |Du|(\Omega_\delta)+|Du|(\Omega_\delta) - \int_{\Omega_\delta} |\nabla |u_p|^{p-1}u_p| +\int_{\Omega_\delta} |\nabla |u_p|^{p-1}u_p|
\\
&\le 2 |Du| (\Omega'_\delta) + \left(|Du|(\Omega_\delta) - \int_{\Omega_\delta} |\nabla |u_p|^{p-1}u_p|\right) + \int_{\Omega} |\nabla |u_p|^{p-1}u_p| 
\\
&-\int_\Omega |\nabla u_p|^p + |\beta|(c_2+ c_1c\delta^{-1})\int_{\Omega'_\delta}|u-|u_p|^{p-1}u_p |.
\end{aligned}
\end{equation}
Let us observe that from the Young inequality one has that 
\begin{equation}\label{betaneg3}
	\int_{\Omega} |\nabla |u_p|^{p-1}u_p| = \int_{\Omega} p|u_p|^{p-1}|\nabla u_p| \le \int_\Omega |\nabla u_p|^p + (p-1)\int_\Omega |u_p|^p . 
\end{equation}
Therefore, gathering \eqref{betaneg3} and \eqref{betaneg2} in \eqref{betaneg1} one yields to 
\begin{equation}\label{betaneg4}
	\begin{aligned}
		J(u)-J_p(u_p) &\le 2 |Du| (\Omega'_\delta) + |Du|(\Omega_\delta) - \int_{\Omega_\delta} |\nabla |u_p|^{p-1}u_p|
		\\
		&+ (p-1)\int_\Omega |u_p|^p
		+|\beta|(c_2 + c_1c\delta^{-1})\int_{\Omega'_\delta}|u-|u_p|^{p-1}u_p|.
	\end{aligned}
\end{equation}
Now observe that, since from the compact embedding one has that $u_p$ converges to $u$ in $L^{q}(\Omega)$ for any $q<\frac{N}{N-1}$, it also holds that $|u_p|^{p-1}u_p$ converges to $u$ in $L^1(\Omega)$ as $p\to 1^+$. Then, taking $p\to 1^+$ in \eqref{betaneg4} (recall that the second term on the right hand side of \eqref{betaneg4} is lower semicontinuous with respect to the $L^1$-convergence), one gains 
\begin{equation*}\label{betaneg5}
	\begin{aligned}
		\limsup_{p\to 1^+} \left(J(u)-J_p(u_p)\right) \le 2|Du| (\Omega'_\delta).
	\end{aligned}
\end{equation*}
Finally, taking $\delta\to0^+$, one obtains \eqref{liminf}. 
\end{proof}

\begin{cor}
\label{gammaconvcor}
Let $\beta > -1$. Then
\begin{equation}
\label{lim1}
\lim_{p\to 1^{+}} \lambda(\Omega,p,\beta)=\Lambda(\Omega,\beta).
\end{equation}
Moreover,  if $u_{p}\in W^{1,p}(\Omega)$ are minimizers of \eqref{eigjp}, with  $\|u_{p}\|_{L^{p}(\Omega)}=1$, then
\[
u_{p}\to u\;\text{weak}^{*}\text{ in }BV(\Omega)
\]
where $u$ is a minimizer of \eqref{minJ}.
\end{cor}

\begin{proof}
Let $\bar u\in BV(\Omega)$ be a minimizer of \eqref{minJ}.
We have that by Theorem \ref{gammaconvteo} there exists $w_{p}$ converging to $\bar u$ which satisfies \eqref{limsup};
on the other hand, we claim that $u_{p}$  weak* converges in $BV(\Omega)$ to a function $u\in BV(\Omega)$.
Indeed, being 
\[
\limsup_{p\to 1^{+}} J_{p}(u_{p}) \le \limsup_{p\to 1^{+}} J_{p}(w_{p}) \le J(\bar u)=\Lambda(\Omega,\beta),
\]
it holds that $J_{p}(u_{p})\le C$, for any $p>1$, where $C$ does not depend on $p$. This implies that $u_{p}$ is bounded in $BV(\Omega)$. This is obvious if $\beta \ge 0$, while if $\beta <0$ it is a consequence of trace inequality \eqref{trace} that
\[
\int_{\Omega} \left|\nabla u_{p}\right|^{p}dx \le C -\beta \int_{\Omega} 
\left|\nabla(u_{p}^{p})\right|-\beta c_{2},
\] 
and by \eqref{betaneg3} it holds
\[
(1+\beta) \int_{\Omega} \left|\nabla u_{p}\right|^{p}dx \le C-\beta C_{2}-\beta(p-1).
\]
Hence, by compactness we get the claim.

 Then it holds, by Theorem \ref{gammaconvteo} and the fact that $\lambda(\Omega,p,\beta)= J_{p}(u_{p})$, that
\begin{equation*}
\Lambda(\Omega,\beta)=J(\bar u) \le J(u)\le \liminf_{p\to 1^{+}} J_{p}(u_{p}) \le \Lambda(\Omega,\beta).
\end{equation*}
Hence \eqref{lim1} holds; moreover, $u$ is a minimizer of \eqref{minJ}.
\end{proof}

\section{An isoperimetric inequality for $\Lambda(\Omega,\beta)$}
\label{s_FK}

In this section we deal with the shape optimization problem; the essential tool is the study of the radial case. In particular we need an explicit computation for $\lambda$ in case of a ball of radius $R$. We first focus on some qualitative properties of the first eigenfunction associated to $\Lambda(\Omega,\beta)$ for any $p>1$. We state the following classical result:  
\begin{prop}
\label{propradiale}
Let  $p>1$, $\beta \in \R$ and let $\Omega=B_{R}$ be the ball in $\R^{N}$ centered at the origin with radius $R$. Let $v_{p}\in W^{1,p}(B_{R})$ be the first positive eigenfunction corresponding to $\lambda(B_{R},p,\beta)$. Then $v_{p}=\psi_{p}(r)\in C^{\infty}((0,R))\cap C^{1,\alpha}([0,R])$ is a radially symmetric function, which solves
\begin{equation*}
\begin{cases}
-\left|\psi_{p}'(r)\right|^{p-2}\left[(p-1)\psi_{p}''(r)+\dfrac{N-1}{r}\psi_{p}'(r)\right]=\lambda(B_{R},p,\beta)\psi_{p}^{p-1}(r) &\text{in }]0,R[,\\[.3cm]
\psi_{p}'(0)=0,\quad \left|\psi_{p}'(R)\right|=\left|\beta\right|^{\frac{1}{p-1}}\psi_{p}(R)
\end{cases}
\end{equation*}
with $\psi_{p}'<0$ in $]0,R]$ if $\beta>0$, $\psi_{p}'>0$ in $]0,R]$ if $\beta<0$.
\end{prop}

Let us state and prove an explicit computation for $\Lambda(\Omega,\beta)$ in case $\Omega$ is a ball. The proof uses classical tools and the reduction of the study of the minimum in \eqref{minJ} to the characteristic functions in case of positive $\beta$. When $\beta$ is negative the proof is more delicate and it strongly relies on the $\Gamma$-convergence proven in Theorem \ref{gammaconvteo}.   
Let us recall that, in the sequel, $h(\Omega)$ is the Cheeger constant for $\Omega$ as defined in \eqref{cheeger}.
\begin{prop}\label{prop_radial}
If $\beta> -1$, then
\begin{equation}
	\label{casopalla}
	\Lambda(B_{R},\beta)=\hat \beta h(B_{R})= \hat \beta \frac{N}{R},
\end{equation}
where $\hat \beta=\min\{\beta,1\}$.
\end{prop}
\begin{proof}
We split the proof in two cases with respect to the sign of $\beta$. 
\medskip

If $\beta \ge 0$, the equality \eqref{casopalla} follows directly from \eqref{char0} and from the isoperimetric inequality. Indeed if $E\subseteq B_{R}$, then
\begin{equation}
\begin{aligned}
\label{cheegerpalla}
R(E,\beta)&=
\frac{P_{B_{R}}(E) + \hat \beta 
\mathcal{H}^{N-1}(\partial B_{R} \cap \partial E)}{|E|}  
\\
&\ge \hat\beta
\frac{P(E)}{\left|E\right|} \ge 
\hat \beta \frac{ P(E^{\#}) }{|E^{\#}|} \ge 
\hat \beta \frac{ P(B_{R}) }{|B_{R}|}=\hat \beta\frac{N}{R},
\end{aligned}
\end{equation}
where the last inequality in the previous holds from the definition and the ball's property of being self-Cheeger given by \eqref{cheeger} and \eqref{cheeger2}. This shows that $\Lambda(B_{R},\beta)\ge \hat \beta\frac{N}{R}$. For the reverse inequality one can simply choose $E=B_{R}$ in \eqref{elle} which gives that \eqref{casopalla} holds.

\medskip

The proof of case $-1<\beta<0$ is more involved and it makes use of the $\Gamma$-convergence of the functional and the Proposition \ref{propradiale}. 

Let $v_{p}\in W^{1,p}(B_{R})$ be a positive minimizer of \eqref{eigjp}. Then it follows from  Corollary \ref{gammaconvcor} that
\[
\Lambda(B_{R},\beta)=\lim_{p\to 1^{+}}J_{p}(v_{p}).
\]
By Proposition \ref{propradiale}, $v_{p}(x)=\psi_{p}(r)$, $r=\left|x\right|$, are radially increasing functions. Without loss of generality, we may suppose $\psi_{p}(R)=1$. 

Moreover, $v_{p}$ converges strongly in $L^{1}(B_{R})$ to $v\in BV(B_{R})$, almost everywhere in $B_{R}$ and weak$^{*}$ in $BV(B_{R})$ as $p\to 1^+$. Being $v_{p}$ radially increasing, so $v$ is nondecreasing. 
 Hence, its superlevel sets $\{v>t\}$ are $N$-dimensional spherical shells $\{r<\left|x\right|<R\}$, and by \eqref{char3} it holds that
\[
\Lambda(B_{R},\beta)=\frac{N}{R}\cdot\frac{\left(\frac{r}{R}\right)^{N-1}+\beta}{1-\left(\frac{r}{R}\right)^{N}}.
\]
for some $r\in [0,R[$. Then minimizing the function
\[
f(t)=\frac{t^{N-1}+\beta}{1-t^{N}},\quad t\in [0,1[
\]
it is easy to see that the minimum is attained at $t=0$, and this gives \eqref{casopalla}.
\end{proof}

Now we are in position to state and prove the above cited inequalities. The proof consists in a suitable use of the explicit computation of $\lambda$ given in Proposition \ref{prop_radial}. We give the following result:
\begin{theorem}\label{teo_fabkra}
It holds 
\begin{equation}
\label{fk1}
\Lambda(\Omega\diesis,\beta) \le \Lambda(\Omega,\beta)\quad \text{if }\beta\ge 0,
\end{equation}
and
\begin{equation}
\label{fk2}
\Lambda(\Omega\diesis,\beta)\ge \Lambda(\Omega,\beta)\quad \text{if }-1<\beta<0.
\end{equation}
\end{theorem}
\begin{proof}

We split the proof depending on the sign of $\beta$. 
In case $\beta \ge 0$, arguing as in \eqref{cheegerpalla} and if $E\subseteq \Omega$, one can write
\[
R(E,\beta) \ge  \Lambda(\Omega\diesis,\beta) =  \hat \beta  \frac{P(\Omega^{\#})}{\left|\Omega^{\#}\right|}
\]
from which it simply follows \eqref{fk1}.

\medskip

If $-1< \beta<0$ then one can use the isoperimetric inequality and \eqref{casopalla} in order to deduce:
\[
\Lambda(\Omega,\beta) \le \beta \frac{P(\Omega)}{\left|\Omega\right|} \le \beta \frac{P(\Omega^{\#})}{\left|\Omega^{\#}\right|}=\Lambda(\Omega\diesis,\beta),
\]
that is \eqref{fk2}.
\end{proof}

\section{A Cheeger-type inequality for p-Laplace with Robin conditions}
\label{s_Cheeger}
As final remark, we derive an estimate for positive $\beta$ (see also \cite{gt} for a similar inequality). This is a generalization, for Robin boundary conditions, of the well-known Cheeger inequality in the Dirichlet ($\beta=+\infty)$ case (see e.g. \cite{kf})
\begin{lemma}\label{cheegertype}
	Let $\beta > 0$, $p>1$ and $\Omega$ Lipschitz bounded open set. Then 
	\[
	\lambda(\Omega,p,\beta) \ge \textcolor{red}{\Lambda(\Omega,\beta)}\tilde \beta -(p-1)\tilde \beta^{\frac{p}{p-1}}
	\]
	where $\tilde\beta=\max\{1,\beta\}$.
Moreover if $\beta\ge \left(\frac{h(\Omega)}{p}\right)^{p-1}$ then  
\[
\lambda(\Omega,p,\beta) \ge \left(\frac{h(\Omega)}{p}\right)^p.
\]
\end{lemma}
\begin{proof}
 Let $\varphi_p$ be a positive minimizer of \eqref{eigjp} 
 and let us denote by
	\[
	U_t=\{x\in \Omega: \varphi_p(x)>t\},
	\]
	\[
	S_t=\{x\in \Omega: \varphi_p(x)=t\},
	\]
	\[
	\Gamma_t=\{x\in \partial\Omega: \varphi_p(x)>t\}.
	\]
	For a continuous function $\psi$ we also define the following functional 
	\[
	H(U_t,\psi) = \frac{1}{|U_t|}\left(\int_{S_t}\psi d\mathcal{H}^{N-1} + \int_{\Gamma_t}\beta d\mathcal{H}^{N-1} - (p-1)\int_{U_t}\psi^{\frac{p}{p-1}} \right).
	\]
	From Proposition $2.3$ and Theorem $3.2$ of \cite{buda}, it follows that there exists a set $S\subseteq (0,1)$ of positive measure such that
	\begin{equation}\label{stimabuda}
		\lambda(\Omega,p,\beta) \ge H(U_t,\psi)	
	\end{equation}
	for all $t\in S$. 
	We also mention that the choice $\psi = \left(\frac{|\nabla \varphi_p|}{\varphi_p}\right)^{p-1}$ gives the equality in \eqref{stimabuda}. 
	\\
	Then choosing $\psi=\tilde\beta$ in \eqref{stimabuda} one gets
	\begin{equation*}
		\lambda(\Omega,p,\beta) \ge H\left(U_t,\psi\right) \ge
		\begin{cases}
		 h(\Omega)\beta-(p-1)\beta^{\frac{p}{p-1}}&\text{if }\beta \ge 1,\\
		 \ds\min_{E\subseteq\Omega}\frac{P_{\Omega}(E)+\beta\mathcal H^{n-1}(\de E\cap \de\Omega)}{\left|E\right|}-(p-1). &\text{if }\beta <1.
		\end{cases}		
	\end{equation*}
	where $h(\Omega)$ is the Cheeger constant for $\Omega$. Recalling that if $\beta\ge 1$ then $\Lambda(\Omega,\beta)=h(\Omega)$, we have
	\[
	\lambda(\Omega,p,\beta) \ge \lambda(\Omega,1,\beta)\tilde\beta -(p-1)\tilde\beta^{\frac{p}{p-1}}.
	\]
	Furthermore, if $\displaystyle \beta\ge \left(\frac{h(\Omega)}{p}\right)^{p-1}$ one gets that
	\[
	\lambda(\Omega,p,\beta) \ge H\left(U_t,\frac{h(\Omega)^{p-1}}{p^{p-1}}\right) \ge 
	\frac{h(\Omega)^p}{p^{p-1}} - (p-1)\frac{h(\Omega)^p}{p^p} = \left(\frac{h(\Omega)}{p}\right)^p,
	\]
	which concludes the proof.
\end{proof}

\section*{Acnowledgements}
This work has been partially supported by PRIN project 2017JPCAPN (Italy) grant: ``Qualitative and quantitative aspects of nonlinear PDEs'', PON Ricerca e Innovazione 2014-2020, by the FRA Project (Compagnia di San Paolo and Universit\`a degli studi di Napoli Federico II) \verb|000022--ALTRI_CDA_75_2021_FRA_PASSARELLI|, and by GNAMPA of INdAM.

\end{document}